\documentclass[a4paper,10pt,reqno]{amsart}
\usepackage[body={13cm,21cm}]{geometry}
\usepackage[initials]{amsrefs}

\usepackage{amssymb,amscd}
\usepackage[all]{xy}
\usepackage{nicefrac}
\usepackage{units}

\usepackage[mathscr]{eucal}
\usepackage{enumerate}
\numberwithin{equation}{section}
\theoremstyle{plain}
\newtheorem{thm}{Theorem}[section]
\newtheorem{cor}[thm]{Corollary}
\newtheorem{lem}[thm]{Lemma}
\newtheorem{prop}[thm]{Proposition}

\newtheorem{exa}[thm]{Example}
\newtheorem*{problem*}{Problem}
\theoremstyle{definition}
\newtheorem{Def}[thm]{Definition}
\theoremstyle{remark}

\newtheorem*{claim*}{Claim}
\newtheorem*{exa*}{Example}
\newtheorem*{rem*}{Remark}
\newtheorem*{rems*}{Remarks}
\newtheorem*{fact*}{Fact}
\newtheorem*{Def*}{Definition}
\DeclareMathOperator{\Gal}{Gal}

\def\Q{{\mathbb Q}}
\def\Z{{\mathbb Z}}

\def\G{{\mathbb G}}

\def\Gal{{\rm Gal}}

\def\ker {{\rm  Ker}}

\DeclareFontFamily{U}{wncy}{}
\DeclareFontShape{U}{wncy}{m}{n}{%
<5>wncyr5%
<6>wncyr6%
<7>wncyr7%
<8>wncyr8%
<9>wncyr9%
<10>wncyr10%
<11>wncyr10%
<12>wncyr6%
<14>wncyr7%
<17>wncyr8%
<20>wncyr10%
<25>wncyr10}{}
\DeclareMathAlphabet{\cyr}{U}{wncy}{m}{n}
\begin{document}
%-----------------------------------------------------------------------%

\title[The sum of two integral squares] {On the sum of two integral squares in certain
quadratic fields}
\author{Dasheng Wei}

\address{ Academy of Mathematics and System Science, CAS, Beijing
100190, P.R. China and Mathematisches Institut der Universit\"at M\"unchen Theresienstr. 39, D-80333 M\"unchen}

\email{dshwei@amss.ac.cn}

\date{\today}

\maketitle{}

\bigskip
%***************************************************************************
\section*{\it Abstract}

In this note, we give a necessary and sufficient condition for determining which integers can be written as
a sum of two integral squares for  certain quadratic fields by using the integral Brauer-manin obstruction (see \cite{CTX}).
The condition is computable and originally from the reciprocity law.

%****************************************************************************
\bigskip

{\it MSC classification} : 11E12;11D09

%****************************************************************************

\bigskip

{\it Keywords} : integral points, ring class field, reciprocity law.

\section*{Introduction} \label{sec.notation}

It is significantly more difficult to study sums of two integral
squares over algebraic number fields than that over $\Z$. From
nowadays point of view, Fermat-Gauss' theorem about sums of two
squares over $\Z$ is a purely local problem. However the question
is a global problem over algebraic number fields, even quadratic
fields since the class number of number fields is involved. There
are only a few results about the question for general algebraic
number fields $F$. Niven studied the problem for $F=\Q({\sqrt{-1}})$
in \cite{Niv}. This case is very special since the binary quadratic
form $x^2+y^2$ is hyperbolic over $\Q(\sqrt{-1})$. Nagell further
studied the problem for $F=\Q(\sqrt{d})$ when $d$ is one of the
following twenty integers:
$$\pm 2,\pm 3,\pm 5,\pm 7,\pm 11,\pm 13,\pm 19,\pm 43,\pm 67,\pm
163$$ in \cite{Nag1} and \cite{Nag2}. His method follows Gauss'
original idea  and essentially depends on the fact that the class
number of $\Q(\sqrt{d},\sqrt{-d})$ is 1 when $d$ is one of the above
integers. Obviously this method cannot be applied to general algebraic
number fields.

Recently Harari (in \cite{Ha08})  proved that the Brauer-Manin
obstruction to the existence of an integral point is the only
obstruction for an integral model of a principal homogenous space of
tori. However, the Brauer group is infinite. One cannot use these
results to determine the existence of integral points on a specific
example. Fei Xu and the author gave another proof of the result in
\cite{WX} and \cite{WX2}. In this paper we apply the method in
\cite{WX} for the sum of two squares over quadratic fields.

It should be pointed out that the method in \cite{WX} only produces
the idelic class groups of $F(\sqrt{-1})$ for solving the problem of
sum of two squares, where these idelic class groups are not unique.
In other word, the finite subgroups of the Brauer group for testing the existence of the integral points are not unique. In
order to get the explicit conditions for the sum of two squares, one
needs further to construct the explicit abelian extensions of
$F(\sqrt{-1})$ corresponding to the idelic class groups by class
field theory. Such explicit construction is a wide open problem
(Hilbert's 12-th problem) in general but ad hoc methods exist.

Notation and terminology are standard if not explained. Let $F$ be a
number field, $\frak o_F$ the ring of integers of $F$, $\Omega_F$
the set of all primes in $F$ and $\infty$ the set of all infinite
primes in $F$. For simplicity, we write $\frak p<\infty$ for $\frak
p\in \Omega_F\setminus \infty$. Let $F_\frak p$ be the completion of
$F$ at $\frak p$ and $\frak o_{F_\frak p}$ be the local completion
of $\frak o_F$ at $\frak p$ for each $\frak p\in \Omega_F$. Write
$\frak o_{F_\frak p}=F_\frak p$ for $\frak p\in \infty$. We also
denote the adele ring (resp. the idele ring) of $F$ by $\Bbb A_F$
(resp. $\Bbb I_F$).

Suppose that $-1$ is not a square in $F$. Let $E=F(\sqrt{-1})$. Let
$\bold X_\alpha$ denote the affine scheme over $\frak o_{F}$ defined
by the equation $x^2+y^2=\alpha$ for a non-zero integer $\alpha \in
\frak o_F$. The equation $x^2+y^2=\alpha$ is solvable over $\frak
o_F$ if and only if $\bold X_\alpha(\frak o_F)\neq \emptyset$. Let
$X_\alpha=\bold X_\alpha \times_{\frak o_{F}} {F}$. Obviously
$f=x+y\sqrt{-1}$ is an invertible function on $X_\alpha\times_F E$,
and $f$ induces a natural map
$$f_E:\ \ \ X_\alpha(\Bbb A_F)\rightarrow \Bbb I_E.$$ The
restriction to $X_\alpha(F_\frak p)$ of $f_E$ can be defined by
$$f_E[(x_\frak p,y_\frak p)]= \begin{cases} (x_\frak p+y_\frak p \sqrt{-1}, x_\frak p-y_\frak p\sqrt{-1})\in E_{v_1}\oplus E_{v_2} \ \ \
& \text{if $\frak p$ splits in $E/F$} \\
x_\frak p+y_\frak p\sqrt{-1}\in E_v \ \ \ & \text{otherwise,}
\end{cases} $$
where $v_1$ and $v_2$ (resp. $v$) are places of $E$ above $\frak p$.
\begin{Def}
Let $K_1,\cdots,K_n$ be finite abelian extensions over $E$. Let
$$\psi_{K_i/E}: \Bbb I_E\rightarrow \Gal(K_i/E) \text{ for } 1\leq
i\leq n$$ be the Artin map. We say that $\alpha$ satisfies the Artin
condition of $K_1,\cdots,K_n$ if there is
$$\prod_{\frak p\leq \infty}(x_\frak p,y_\frak p)\in \prod_{\frak p\leq \infty}
\bold X_\alpha(\frak o_{F_\frak p})$$such that
$$\psi_{K_i/E}(f_E[\prod_{\frak p\leq \infty}(x_\frak p,y_\frak p)])=1_{i} \text{ for } i=1,\cdots,n$$
where $1_i$ is the identity element of $\Gal(K_i/E)$.
\end{Def}

By class field theory, it is a necessary condition for $\bold
X_\alpha(\frak o_F) \neq \emptyset$ that $\alpha$ satisfies the
Artin condition of $K_1,\cdots,K_n$. There exists a (non-unique)
finite abelian extension $K/E$ independent of $\alpha$, such that
the Artin condition of $K$ is also sufficient for $\bold
X_\alpha(\frak o_F) \neq \emptyset$ (Corollary 1.6 and Theorem 1.10 in \cite{WX}).
Let $d\geq 2$ be a positive square-free integer and
$F=\Q(\sqrt{-d})$. Then the field $K$ closely depends on the (local
and global) solvability of the following equation
\begin{equation} \label{equ-1}x^2+y^2=-1.
\end{equation} Let $L=\frak o_{F}+\frak o_{F}\sqrt{-1}$ and
$H_L$ be the ring class field corresponding to the order $L$. For
example, the Artin condition of $H_L$ is sufficient if
$F=\Q(\sqrt{p}) \text{ or } \Q(\sqrt{-p})$ with $p$ prime (Theorem
0.3 in \cite{Wei}), and some related works were also performed in \cite{Wei1} and \cite{Wei2}. The following result can be found in
\cite{Wei} (Proposition 1.1).

\begin{prop} \label{ima}
Suppose one of the following conditions holds:

(1) The equation (\ref{equ-1}) has an integral solution in $\frak
o_F$.

(2) The equation (\ref{equ-1}) has no local integral solutions at a
place of $F$.

Then the equation $x^2+y^2=\alpha$ is solvable over $\frak o_F$ if
and only if $\alpha$ satisfies the Artin condition of $H_L$.
\end{prop}

Therefore we only need to consider the case that the two conditions
in Proposition \ref{ima} do not hold. For this case the Artin
condition of $H_L$ is not sufficient since $-1$ satisfies the Artin
condition of $H_L$ but the equation (\ref{equ-1}) is not solvable. In general the question
will be very complicated. We cannot expect to construct a field $K$
generally such that the Artin condition of $K$ is sufficient.

The following result is well-known (Satz 2, \cite{Pe}): the equation
(\ref{equ-1}) is solvable over $\frak o_F$ if and only if the
equation
\begin{equation}\label{equ-2} x^2-dy^2=-\gamma(d) \end{equation} is solvable over $\Z$, where
$$\gamma(d)=\begin{cases} 1
  &\text{if   } d\not \equiv -1 \mod 4 \\ 2 &\text{if   } d\equiv -1 \mod 4.\end{cases}$$
In this paper, we will consider the case that the equation
(\ref{equ-1}) is solvable over $\frak o_{F_\frak p}$ for any $\frak
p\in \Omega_F$ and the equation (\ref{equ-2}) is not solvable over
$\Z_p$ for some prime $p$. It is easy to see that the two conditions
in Proposition \ref{ima} don't hold for this case. Let
$$\aligned& C=\{(d,p)|d \text{ is a square-free positive integer and }p\mid
d \text{ with }p \text{ prime}\}\\
&D_1=\{(d,p)\in C| d\not \equiv -1 \mod 8, p\equiv -1 \mod 8\}\\
&D_2=\{(d,p)\in C| d\equiv 1, 2 \mod 4, p\equiv 3 \mod 8\}\\
&D_3=\{(d,p)\in C| d\equiv 3 \mod 8, p\equiv 5 \mod 8\}\endaligned$$
Denote $$D=D_1\cup D_2\cup D_3.$$ We say that $d\in D\ (\text{or }
D_i)$ if there is a prime $p$ such that $(d,p)\in D\ (\text{or }
D_i)$. It is easy to verify that the equation (\ref{equ-1}) is
solvable over $\frak o_{F_\frak p}$ for any $\frak p\in \Omega_F$
and the equation (\ref{equ-2}) is not solvable over $\Z_p$ for some
prime $p$ if and only if $d\in D$. In this paper, we mainly prove
the following result.
\begin{thm}\label{0.1} Let $(d,p)\in D$ and $F=\Q(\sqrt{-d})$.
Then the diophantine equation $x^2+y^2=\alpha$ is solvable over
$\frak o_F$ if and only if $\alpha$ satisfies the Artin condition of
$\Theta$ and $H_L$,where $\Theta=E(\sqrt[4]{p})$.
\end{thm}

Sums of two integral squares over real quadratic fields are more
complicated than that over imaginary quadratic fields, since there
are infinitely many units in real quadratic fields. In this paper, we also
obtain a result about real quadratic fields.
\begin{thm}\label{0.2}  Let $p$ be a prime and $F=\Q(\sqrt{2p})$. If $p\equiv 3 \mod
8$, then the diophantine equation $x^2+y^2=\alpha$ is solvable over
$\frak o_F$ if and only if $\alpha$ satisfies the Artin condition of
$\Theta$ and $H_L$, where $\Theta=E(\sqrt[4]{2})$.
\end{thm}
As application, we explicitly determine which integers can be
written as a sum of two integral squares for the two quadratic
fields $\Q(\sqrt{\pm 6})$. Note that $\Q(\sqrt{\pm 6})$ are the fields
with the smallest $|d|$ that the Artin condition of $H_L$ is not
sufficient for $\Q(\sqrt{d})$.

\section{The sum of two squares in imaginary quadratic fields }

Let $F$ be an algebraic number field such that $-1$ is not a square in
$F$. Let $E=F(\sqrt{-1})$ and let $T$ be the torus $R^1_{E/F}(\Bbb
G_m)=\ker[R_{E/F}(\Bbb G_{m,E})\rightarrow \Bbb G_{m,F}]$. Here $R_{E/F}$
denotes the Weil's restriction (see \cite{Milne98}). Denote
$\lambda$ to be the embedding from $T$ to $ R_{E/F}(\Bbb G_{m,E})$.
Obviously $\lambda$ induces a natural group homomorphism
$$\lambda_E:\ \ \ T(\Bbb A_F)\rightarrow \Bbb I_E.$$  Let $\bold X_\alpha$ denote the affine scheme over
$\frak o_{F}$ defined by $x^2+y^2=\alpha$ for a non-zero integer
$\alpha \in \frak o_F$. Let $\bold T$ be the group scheme over
$\frak o_{F}$ defined by $x^2+y^2=1$ and let $T=\bold T
\times_{\frak o_{F}} {F}$. The generic fiber of $\bold X_\alpha$ is
a principal homogenous space of the torus $T$. Since $\bold T$ is
separated over $\frak o_F$, we can view $\bold T(\frak o_{F_\frak
p})$ as a subgroup of $T(F_{\frak p})$. The following result is similar to \cite[Corollary 1.6]{WX}.

\begin{prop} \label{multiple} Let $K_1/E$ and $K_2/E$ be finite abelian
extensions
 such that the group homomorphism $\widetilde{\lambda}_E$ induced by
$\lambda_E$
$$\widetilde{\lambda}_E: \nicefrac{T(\Bbb A_F)}{T(F)\prod_{\frak p\leq \infty}\bold T(\frak o_{F_\frak p})}
\longrightarrow \nicefrac{\Bbb I_{E}}{E^*N_{K_1/E}(\Bbb
I_{K_1})}\times \nicefrac{\Bbb I_{E}}{E^*N_{K_2/E}(\Bbb I_{K_2})}
$$ is well-defined and injective, where well-defined means
$$ \lambda_E(T(F)\prod_{\frak p\leq \infty}\bold T(\frak o_{F_\frak p}))\subset (E^*N_{K_1/E}(\Bbb I_{K_1}))\cap (E^*N_{K_2/E}(\Bbb I_{K_2})).$$
Then $\bold X_\alpha(\frak o_F)\neq \emptyset$ if and only if
$\alpha$ satisfies the Artin condition of $K_1$ and $K_2$.
\end{prop}
Let $K$ be the composite field of $K_1$ and $K_2$,
in fact the Artin condition of $K$ is equivalent with that of $K_1$ and $K_2$. In this note, $K_1$ and $K_2$ are disjoint over $E$ and have canonical definitions, hence we can compute the Artin condition of $K$  by computing the Artin character of $K_1/E$ and $K_2/E$ respectively.

Let $d\geq 2$ be a square-free positive integer. Let
$F=\Q(\sqrt{-d})$, $\frak o_{F}$ be the ring of integers of $F$ and
$E=F(\sqrt{-1})$. In the rest of this section $F$ is always an
imaginary quadratic field. One takes the order $L= \frak o_{F}+
\frak o_{F} \sqrt{-1}$ inside $E$. Let $H_L$ be the ring class field
corresponding to the order $L=\frak o_{F}+ \frak o_{F} \sqrt{-1}$.
Recall some notations:
$$\aligned& C=\{(d,p)|d \text{ is a square-free positive integer and }p\mid
d \text{ with }p \text{ prime}\}\\
&D_1=\{(d,p)\in C| d\not \equiv -1 \mod 8, p\equiv -1 \mod 8\}\\
&D_2=\{(d,p)\in C| d\equiv 1, 2 \mod 4, p\equiv 3 \mod 8\}\\
&D_3=\{(d,p)\in C| d\equiv 3 \mod 8, p\equiv 5 \mod 8\}.
\endaligned$$ Note that $p\equiv 3 \mod 4$ if $(d,p)\in D_1\cup
D_2$.

\begin{prop} \label{computation-1} Suppose $(d,p)\in D_1\cup D_2$ and
$F=\Q(\sqrt{-d})$. Let $(x_\frak p,y_\frak p)\in \frak o_{F_\frak
p}\times \frak o_{F_\frak p}$, then the $4$-th Hilbert symbol
$$\left(\frac{x_\frak p+y_\frak p\sqrt{-1}, p}{v}\right)_{4}=\begin{cases}1 &\text{if } x_\frak p^2+y_\frak p^2=1\\
-1 &\text{if } x_\frak p^2+y_\frak p^2=-1,\end{cases}$$ where $v$
and $\frak p$ are respectively  the unique place of $E$ and $F$
above $p$.
\end{prop}
\begin{proof} The Hilbert symbol
$$\aligned\left(\frac{x_\frak p+y_\frak p\sqrt{-1}, p}{v}\right)_{4}
& = \left(\frac{x_\frak p+y_\frak p\sqrt{-1},-d}{v}\right)_{4}\cdot\left(\frac{x_\frak p+y_\frak p\sqrt{-1},-d/p}{v}\right)^{-1}_{4}\\
& = \left(\frac{x_\frak p+y_\frak p\sqrt{-1},-d}{v}\right)_{4}\cdot 1=\left(\frac{x_\frak p+y_\frak p \sqrt{-1},\sqrt{-d}}{v}\right)\\
& = \left(\frac{\pm 1,p}{p}\right)=\pm 1,
\endaligned$$
where the second equation holds since $E(\sqrt[4]{-d/p})/E$ is
unramified at $v$.
\end{proof}

\begin{prop} \label{computation-2} Suppose $(d,p)\in D_1\cup D_2$ and
$F=\Q(\sqrt{-d})$. If $x_\frak p$ and $y_\frak p$ in $\frak
o_{F_\frak p}$ satisfy $x_\frak p^2+y_\frak p^2=\pm 1$, then
$$\prod_{v|2} \left(\frac{x_\frak p+y_\frak p\sqrt{-1}, p}{v}\right)_{4}=1$$
where $v\in \Omega_E$ and $\frak p$ is the unique place of $F$ above
$2$.
\end{prop}
The proposition follows from the following series of lemmas.  Note that we do not assume that
$d>0$, $l$ prime and $l\mid d$ in the following lemmas and some lemmas  will be used in the real quadratic field case.
\begin{lem} \label{p}Let $d\not \equiv -1 \mod 8$ and $l\equiv -1\mod
8$. Let $F=\Q(\sqrt{-d})$. If $x_\frak p$ and $y_\frak p$ in
$F_\frak p$ satisfy $x_\frak p^2+y_\frak p^2=\pm 1$, then
$$\prod_{v|2}\left(\frac{x_\frak p+y_\frak p\sqrt{-1}, l}{v}\right)_{4}=1$$
where $v\in \Omega_E$ and $\frak p$ is the unique place of $F$ above
$2$.
\end{lem}
\begin{proof}The extension $E(\sqrt{l})/E$ is split over $v$, hence
$$\sqrt{l}=s\sqrt{-1}=(2\sqrt{-1})\cdot\frac{s}{2}=(1+\sqrt{-1})^2\frac{2s}{2^2},$$
where $s^2=-l$ and $s\in \Z^\times_2$. For any $x_\frak p$ and
$y_\frak p$ in $F_\frak p$ satisfy $x_\frak p^2+y_\frak p^2=\pm1$,
we have
$$\aligned \prod_{v|2}\left(\frac{x_\frak p+y_\frak p\sqrt{-1},l}{v}\right)_{4}
&=\prod_{v|2}\left(\frac{x_\frak p+y_\frak
p\sqrt{-1},\sqrt{l}}{v}\right)=\prod_{v|2}\left(\frac{x_\frak
p+y_\frak p\sqrt{-1},2s}{v}\right)\\
&=\left(\frac{\pm 1,2s}{\frak p}\right)=1.
\endaligned$$
\end{proof}

Proposition \ref{computation-2} for the case $(d,p)\in D_1$ follows
from the above lemma. The following two lemmas deal with the case
$(d,p)\in D_2$.
\begin{lem} \label{3} Let $d\equiv 1 \mod 4$ and $l\equiv 3\mod 8$. Let $F=\Q(\sqrt{-d})$.
If $x_\frak p$ and $y_\frak p$ in $\frak o_{F_\frak p}$ satisfy
$x_\frak p^2+y_\frak p^2=\pm 1$, then
$$\prod_{v|2}\left(\frac{x_\frak p+y_\frak p\sqrt{-1}, l}{v}\right)_{4}=1$$
where $v\in \Omega_E$ and $\frak p$ is the unique place of $F$ above
$2$.
\end{lem}
\begin{proof}
(1) If $d\equiv 5\mod 8$, then $F=\Q(\sqrt{-d})=\Q(\sqrt{l})$ and
$v$ is the unique place of $E$ above $2$. For any $x_\frak p$ and
$y_\frak p$ in $F_\frak p$ satisfy $x_\frak p^2+y_\frak p^2=\pm1$,
we have
$$ \aligned \left(\frac{x_\frak p+y_\frak p\sqrt{-1},l}{v}\right)_{4}
& =\left(\frac{x_\frak p+y_\frak p\sqrt{-1},\sqrt{l}}{v}\right)\\
&=\left(\frac{\pm 1,\sqrt{l}}{\frak p}\right) =\left(\frac{\pm 1,-l}{2}\right) \\
& =1. \endaligned$$

(2) If $d\equiv 1\mod 8$, then $E/F$ is split over $\frak p$. For
any $x_\frak p$ and $y_\frak p$ in $\frak o_{F_\frak p}$ satisfy
$x_\frak p^2+y_\frak p^2=\pm1$, we have $$\left(\frac{x_\frak
p+y_\frak p\sqrt{-1},l}{v}\right)=1$$ since $E(\sqrt{l})/E$ is
unramified at any place $v$ of $E$ above $2$. So we have
$$\left(\frac{x_\frak p+y_\frak p\sqrt{-1},l}{v}\right)_{4}=\pm 1.$$ Therefore
$$\aligned \prod_{v|2}\left(\frac{x_\frak p+y_\frak p\sqrt{-1},l}{v}\right)_{4}
&=\left(\frac{(x_\frak p+y_\frak p\sqrt{-1})(x_\frak p-y_\frak p\sqrt{-1}),l}{\frak p}\right)_{4}\\
&=\begin{cases}1 &\text{ if
 } x_\frak p^2+y_\frak p^2= 1\\ \left(\frac{-1,l}{\frak p}\right)_{4}=\left(\frac{\sqrt{-1},l}{\frak p}\right)=1 & \text{ if
 } x_\frak p^2+y_\frak p^2=-1. \end{cases}\endaligned$$
\end{proof}

\begin{lem} \label{3-2d}Let $d\equiv 2 \mod 4$ and $l\equiv \pm 3\mod
8$. Let $F=\Q(\sqrt{-d})$. If $x_\frak p$ and $y_\frak p$ in $\frak
o_{F_\frak p}$ satisfy $x_\frak p^2+y_\frak p^2=\pm 1$, then
$$\left(\frac{x_\frak p+y_\frak p\sqrt{-1}, l}{v}\right)_{4}=1$$
where $v$ and $\frak p$ are respectively the unique place of $E$ and
$F$ above $2$.
\end{lem}
\begin{proof}(1) First we will prove that $\left(\frac{x_\frak p+y_\frak p\sqrt{-1},
l}{v}\right)_{4}=1$ for any $x_\frak p$ and $y_\frak p$ in $\frak
o_{F_\frak p}$ that satisfy $x_\frak p^2+y_\frak p^2=1$.

By Hilbert 90, there exists $\beta \in E_v^*$ such that $x_\frak
p+y_\frak p\sqrt{-1}=\sigma(\beta)/\beta$, where $\sigma$ is the
non-trivial element of $\Gal(E_v/F_\frak p)$. Let
$\beta=a+b\sqrt{-1}$ and  we can choose $\beta$ such that $a,b \in
\frak o_{F_\frak p}$ and $a$ or $b$ is a unit. Denote
$\kappa=a^2+b^2$.

Assume that $ord_{\frak p}(\kappa)$ is odd. The equation
$x^2+y^2=\kappa$ is not solvable over $\frak o_{F_\frak p}$ if
$ord_{\frak p}(\kappa)=1$ (by Theorem 1 in \cite{Ko}). Therefore
$ord_{\frak p}(\kappa)\geq 3$. We have $a,b \in \frak
o_{F_p}^\times$ and
$$\beta/\sigma(\beta)=(a^2-b^2-2ab\sqrt{-1})\kappa^{-1}=1-2\kappa^{-1}b^2-2\kappa^{-1}ab\sqrt{-1}.$$
Since $ord_{\frak p}(\kappa)\geq 3$ and $a,b \in \frak o_{F_\frak
p}^\times$, one has
$$2\kappa^{-1}b^2,2\kappa^{-1}ab\not \in \frak o_{F_\frak p}.$$
So $\beta/\sigma(\beta)\not\in \frak o_{F_\frak p}+\frak o_{F_\frak
p}\sqrt{-1}$. A contradiction is derived, so one obtains $ord_{\frak
p}(\kappa)$ is even.

Let $ord_{\frak p}(\kappa)=2n$.  Since $E/F$ is totally ramified and
of degree $2$, one can write $\beta=\pi_F^n\mu$, where $\pi_F$ is a
uniformizer of $F$ and $\mu\in \frak o_{E_v}^\times$. Then
$\sigma(\beta)/\beta=\sigma(\mu)/\mu$. Let $\frak q$ be the unique
place of $\Q(\sqrt{-1})$ above $2$.  We have
$$ \aligned \left(\frac{x_{\frak p}+y_{\frak p}\sqrt{-1},l}{v}\right)_{4}
=&\left(\frac{\sigma(\mu)/\mu,l}{v}\right)_{4}=\left(\frac{N_{E_v/F_\frak
p}(\mu),l}{v}\right)_{4}\cdot
\left(\frac{\mu,l}{v}\right)^{-1}\\
&=\left(\frac{N_{E_v/\Q_2}(\mu),l}{\frak q}\right)_{4}\cdot
\left(\frac{N_{E_v/\Q_2}(\mu),l}{2}\right)^{-1}.
\endaligned$$
Since $\mu\in \frak o_{E_v}^\times$ and
$E_v=\Q_2(\sqrt{-1},\sqrt{-d})$ with $d\equiv 2 \mod 4$, one has
$$N_{E_v/\Q_2}(\mu)=m^2 \text{ for some } m\in \Z_2^\times.$$ So
$$ \aligned &\left(\frac{x_{\frak p}+y_{\frak p}\sqrt{-1},l}{v}\right)_{4}
=\left(\frac{m^2,l}{\frak q}\right)_{4}\cdot
1=\left(\frac{m,l}{\frak q}\right)=1.
\endaligned$$

(2) Let $d=2d_0$. Then $d_0$ is odd.  By the argument in (1), we
only need to show that there exists a $(x_\frak p,y_\frak p)\in
\frak o_{F_\frak p}\times \frak o_{F_\frak p}$ satisfying $x_\frak
p^2+y_\frak p^2=-1$, such that
$$\left(\frac{x_\frak p+y_\frak p\sqrt{-1},l}{v}\right)_{4}=1.$$

(i) If $d_0\equiv 1\mod 4$, then $x^2-2d_0 y^2=-1$ is solvable over
$\Z_2$. Choose one solution $(x_0,y_0)\in \Z_2\times \Z_2$, then we
have $x_0^2+(y_0\sqrt{-2d_0})^2=-1$. Let
$$x_\frak p=x_0 \text{ and }y_\frak p=y_0\sqrt{-2d_0}.$$ One obtains
$$ \aligned\left(\frac{x_\frak p+y_\frak p\sqrt{-1},l}{v}\right)_{4}&
=\left(\frac{x_0-y_0\sqrt{2d_0},l}{v}\right)_{4}=\left(\frac{-1,l}{\frak
q}\right)_{4} =\left(\frac{\sqrt{-1},l}{\frak q}\right)=1.
\endaligned$$

Let $s\in \Z_2^\times$ such that $s^2=l/3\text { or }-l/3$. For any
$\delta \in E_v^*$ satisfies $N_{E_v/F_\frak p}(\delta)=\pm 1$, we
have
$$\left(\frac{\delta, l}{v}\right)_{4}=\left(\frac{\delta,3}{v}\right)_{4}\cdot \left(\frac{\delta,s}{v}\right)\cdot \left(\frac{\delta,\pm 1}{v}\right)_{4}
=\left(\frac{\delta,3}{v}\right)_{4}\cdot\left(\frac{\pm 1,s}{\frak
p}\right)\cdot 1=\left(\frac{\delta,3}{v}\right)_{4}.$$

(ii) If $d_0 \equiv 3 \mod 8$, then $F_\frak p=\Q_2(\sqrt{-6})$. Let
$$x_\frak p=(1+2\sqrt{-6})/5 \text{ and }y_\frak p=(2-\sqrt{-6})/5.$$ Then
$x_\frak p^2+y_\frak p^2=-1$. One has
$$ \aligned
\left(\frac{x_\frak p+y_\frak p\sqrt{-1},l}{v}\right)_{4}&=\left(\frac{5^{-1}(1+2\sqrt{-1})+5^{-1}(2-\sqrt{-1})\sqrt{-6},3}{v}\right)_{4}\\
&=\left(\frac{5^{-1}(3-4\sqrt{-1}),3}{\frak q}\right)_{4}\\
&=\left(\frac{3-4\sqrt{-1},3}{\frak q}\right)_{4}\cdot
\left(\frac{5,3}{\frak q}\right)_{4}^{-1}.
\endaligned$$
It is easy to see $$\left(\frac{5,3}{\frak
q}\right)_{4}=\left(\frac{-3,3}{\frak q}\right)_{4}=1.$$ By
class field theory, we have
$$ \aligned \left(\frac{x_\frak p+y_\frak p\sqrt{-1},l}{v}\right)_{4}
&=\prod_{w|3}\left(\frac{3-4\sqrt{-1},3}{w}\right)_{4}\prod_{w|5}\left(\frac{3-4\sqrt{-1},3}{w}\right)_{4}\\
&= (-\sqrt{-1})^{(3^2-1)/4}\cdot \left(\frac{5^2,3}{5}\right)_{4}\\
&=(-1)\cdot (-1)=1,
\endaligned$$
where $w$ is in the set of places of $\Q(\sqrt{-1})$.

(iii) If $d_0 \equiv 7 \mod 8$, then $F_{\frak p}=\Q_2(\sqrt{-14})$.
Let
$$x_\frak p=(3-2\sqrt{-14})/13 \text{ and }y_\frak p=(2+3\sqrt{-14})/13.$$ Then
$x_\frak p^2+y_\frak p^2=-1$. One has
$$ \aligned
\left(\frac{x_\frak p+y_\frak p\sqrt{-1},l}{v}\right)_{4}&=\left(\frac{13^{-1}(3+2\sqrt{-1})+13^{-1}(-2+3\sqrt{-1})\sqrt{-14},3}{v}\right)_{4}\\
&=\left(\frac{13^{-1}(-5-12\sqrt{-1}),3}{\frak q}\right)_{4}\\
&=\left(\frac{-5-12\sqrt{-1},3}{\frak q}\right)_{4}\cdot
\left(\frac{13,3}{\frak q}\right)_{4}^{-1}.
\endaligned$$
It's easy to see $$\left(\frac{13,3}{\frak
q}\right)_{4}=\left(\frac{-3,3}{\frak q}\right)_{4}=1.$$ By
class field theory, we have
$$ \aligned \left(\frac{x_\frak p+y_\frak p\sqrt{-1},l}{v}\right)_{4}
&=\prod_{w|3}\left(\frac{-5-12\sqrt{-1},3}{w}\right)_{4}\prod_{w|13}\left(\frac{-5-12\sqrt{-1},3}{w}\right)_{4}\\
&= 1\cdot \left(\frac{13^2,3}{13}\right)_{4}=1\cdot 1=1,
\endaligned$$
where $w$ is in the set of places of $\Q(\sqrt{-1})$.
\end{proof}

Recall
$$D_3=\{(d,p)\in C| d\equiv 3 \mod 8, p\equiv 5 \mod 8\}.$$ In the
following, we give some properties about the case $(d,p)\in D_3$.
\begin{prop} \label{computation-3}Let $(d,p)\in D_3$ and $F=\Q(\sqrt{-d})$.
If $x_\frak p$ and $y_\frak p$ in $\frak o_{F_\frak p}$ satisfy
$x_\frak p^2+y_\frak p^2=\pm 1$, then
$$\prod_{v|p}\left(\frac{x_\frak p+y_\frak p\sqrt{-1}, p}{v}\right)_{4}=1$$
where $v\in \Omega_E$ and $\frak p$ is the unique place of $F$ above
$p$.
\end{prop}
\begin{proof}The Hilbert symbol
$$\aligned \prod_{v|p}\left(\frac{x_\frak p+y_\frak p\sqrt{-1}, p}{v}\right)_{4}
& = \prod_{v|p}\left(\frac{x_\frak p+y_\frak p\sqrt{-1},-d}{v}\right)_{4}\cdot\prod_{v|p}\left(\frac{x_\frak p+y_\frak p\sqrt{-1},-d/p}{v}\right)^{-1}_{4}\\
& = \prod_{v|p}\left(\frac{x_\frak p+y_\frak p\sqrt{-1},-d}{v}\right)_{4}\cdot 1=\prod_{v|p}\left(\frac{x_\frak p+y_\frak p \sqrt{-1},\sqrt{-d}}{v}\right)\\
& = \left(\frac{\pm 1,p}{p}\right)=1,
\endaligned$$
where the second equation holds since $E(\sqrt[4]{-d/p})/E$ is
unramified over $v$.
\end{proof}

\begin{prop} \label{computation-4}Let $(d,p)\in D_3$ and $F=\Q(\sqrt{-d})$.
Let $(x_\frak p,y_\frak p)\in \frak o_{F_\frak p}\times
\frak o_{F_\frak p}$, then
$$\left(\frac{x_\frak p+y_\frak p\sqrt{-1}, p}{v}\right)_{4}=\begin{cases}1 &\text{if } x_\frak p^2+y_\frak p^2=1\\
-1 &\text{if } x_\frak p^2+y_\frak p^2=-1\end{cases}$$ where $v$ and
$\frak p$ are respectively the unique place of $E$ and $F$ above
$2$.
\end{prop}

\begin{proof}The Hilbert symbol
$$\aligned\left(\frac{x_\frak p+y_\frak p\sqrt{-1}, p}{v}\right)_{4}
& =\left(\frac{x_\frak p+y_\frak p \sqrt{-1},\sqrt{p}}{v}\right)= \left(\frac{\pm 1,\sqrt{p}}{\frak p}\right)\\
& =\left(\frac{\pm 1,-p}{2}\right)= \pm 1,
\endaligned$$
where the second equation holds since $F_\frak
p=\Q_2(\sqrt{-d})=\Q_2(\sqrt{p})$.
\end{proof}

Now we can prove Theorem \ref{0.1} by using the above propositions.
\begin{proof} Let $L=\frak o_F+\frak o_F\sqrt{-1}$.
Let $\frak p$ be a place of $F$ and $L_\frak p$ be the $\frak
p$-adic completion of $L$ inside $E_\frak p = E\otimes_F F_\frak p$.
Recall $T=R^1_{E/F}(\G_{m,E})$ and $\bold T$ is the affine scheme
defined by the equation $x^2+y^2=1$, we have
$$\aligned T(F)&=\{\beta \in E^*: \ N_{E/F}(\beta)=1\}\\
\bold T(\frak o_{F_\frak p})&=\{\beta \in L_\frak p^\times: \
N_{E_\frak p/F_\frak p}(\beta)=1\}.\endaligned$$ And $L_\infty
^\times=E_\infty^*=\Bbb C^*\times \Bbb C^*$.

Let $v\in \Omega_E$ and $v\mid 2$. Let $\frak p$ be the unique place
of $F$ above $2$. By Proposition \ref{computation-2} and
\ref{computation-4}, one has $$\prod_{v\mid 2}\left(\frac{x_\frak
p+y_\frak p\sqrt{-1}, p}{v}\right)_{4}=1$$ for any $(x_\frak
p,y_\frak p)\in \frak o_{F_\frak p}\times \frak o_{F_\frak p}$ with
$x_\frak p^2+y_\frak p^2=1$. Regard $\bold T(\frak o_{F_{\frak p}})$
as a subgroup of $T(\Bbb A_F)$, this implies that
$$\lambda_E(\bold T(\frak o_{F_{\frak p}}))\subseteq E^* N_{\Theta/E}(\Bbb I_\Theta).$$

Let $v\in \Omega_E$ and $v\mid p$. Let $\frak p$ be the unique place
of $F$ above $p$. By Proposition \ref{computation-1} and
\ref{computation-3}, one has $$\prod_{v\mid p}\left(\frac{x_\frak
p+y_\frak p\sqrt{-1}, p}{v}\right)_{4}=1$$ for any $(x_\frak
p,y_\frak p)\in \frak o_{F_\frak p}\times \frak o_{F_\frak p}$ with
$x_\frak p^2+y_\frak p^2=1$. This implies that
$$\lambda_E(\bold T(\frak o_{F_{\frak p}}))\subseteq E^* N_{\Theta/E}(\Bbb I_\Theta).$$

The field extension $\Theta/E$ is unramified over each place $v$ of $E$ except
$v\mid 2p$. Therefore the natural group homomorphism
$$\widetilde{\lambda}_E: \nicefrac{T(\Bbb A_F)}{T(F)\prod_{\frak p\leq \infty}\bold T(\frak o_{F_\frak p})}
\longrightarrow [\nicefrac{\Bbb I_E}{E^*  N_{\Theta/E}(\Bbb
I_{\Theta})}] \times [\nicefrac{\Bbb I_E}{E^* \prod_{\frak p\leq
\infty} L_\frak p^\times}]$$ is well-defined. By Proposition
\ref{multiple}, we only need to show $\widetilde{\lambda}_E$ is
injective.

Let $u\in ker \widetilde{\lambda}_E$. Then there are $\beta\in E^*$
and $i\in \prod_{\frak p\leq \infty} L_\frak p^\times$ with
$\lambda_E(u)=\beta i$. We have
$$N_{E/F}(\beta)=N_{E/F}(i)^{-1} \in F^*\cap
(\prod_{\frak p\leq \infty}\frak o_{F_\frak p}^\times )=\{\pm 1\},
$$ since $F$ is an imaginary quadratic field and $F \neq
\Q(\sqrt{-1})\text { or }\Q( \sqrt {-3})$.

If $N_{E/F}(\beta)\neq 1$, one obtains
$N_{E/F}(\beta)=N_{E/F}(i)=-1$. Write $i=(i_v)_{v}\in \Bbb I_E$.
Since $\Theta/E$ is unramified over each place $v$ of $E$ except
$v\mid 2p$, one concludes that $\psi_{\Theta/E}(i_v)$ is trivial for
all primes $v\nmid 2p$, where $i_v$ is regarded as an idele whose
$v$-component is $i_v$ and 1 otherwise.

(1) Suppose $(d,p)\in D_1\cup D_2$. Since $N_{E/F}(i_v)=-1$ and
$i_v\in L_\frak p^\times$, one gets$$ \prod_{v\mid
2}\psi_{\Theta/E}(i_{v})=1 \text{ and }\psi_{\Theta/E}(i_{v'})=-1$$
by Proposition \ref{computation-1} and \ref{computation-2}, where
 $\psi_{\Theta/E}: \Bbb
I_E\rightarrow \Gal(\Theta/E)$ is the Artin map and $v'$ is the
unique place of $E$ above $p$. So
$$\psi_{\Theta/E}(\beta i)=\psi_{\Theta/E}(i)=\prod_{v\mid 2}\psi_{\Theta/E}(i_{v})\cdot \psi_{\Theta/E}(i_{v'})=-1.$$
This contradicts to $u \in ker \widetilde{\lambda}_E$.

(2) Suppose $(d,p)\in D_3$. Since $N_{E/F}(i_v)=-1$ and $i_v\in
L_\frak p^\times$, one gets $$\psi_{\Theta/E}(i_{v'})=-1 \text{ and
}\prod_{v\mid p}\psi_{\Theta/E}(i_{v})=1$$ by Proposition
\ref{computation-3} and \ref{computation-4}, where $v'$ is the
unique place of $E$ above $2$. So
$$\psi_{\Theta/E}(\beta i)=\psi_{\Theta/E}(i)=\psi_{\Theta/E}(i_{v'})\cdot \prod_{v\mid p}\psi_{\Theta/E}(i_{v})=-1.$$
This contradicts to $u \in ker \widetilde{\lambda}_E$.

Therefore $N_{E/F}(\beta)=1$, one concludes that
$$ N_{E/F}(\beta) =N_{E/F}(i) =1 \ \ \ \Rightarrow \ \ \ \beta\in
T(F) \ \ \ \text{and} \ \ \ i\in \prod_{\frak p\leq \infty}\bold
T(\frak o_{F_\frak p}).$$ So $\beta i\in T(F)\prod_{\frak p\leq
\infty}\bold T(\frak o_{F_\frak p})$. Then $\widetilde{\lambda}_E$
is injective.
\end{proof}

\begin{lem} \label{2} Let $F=\Q(\sqrt{-2d})$ and $d\equiv 3 \mod 4$.
If $x_\frak p$ and $y_\frak p$ in $\frak o_{F_\frak p}$ satisfy
$x_\frak p^2+y_\frak p^2=-1$, then the $4$-th Hilbert symbol
$$\left(\frac{x_\frak p+y_\frak p\sqrt{-1}, 2}{v}\right)_{4}=-1$$
where $v$ and $\frak p$ are respectively  the unique place of $E$
and $F$ above $2$.
\end{lem}
\begin{proof} The Hilbert symbol
$$\aligned\left(\frac{x_\frak p+y_\frak p\sqrt{-1}, 2}{v}\right)_{4}
& = \left(\frac{x_\frak p+y_\frak p\sqrt{-1},-2d}{v}\right)_{4}\cdot\left(\frac{x_\frak p+y_\frak p\sqrt{-1},-d}{v}\right)^{-1}_{4}\\
& = \left(\frac{x_\frak p+y_\frak p\sqrt{-1},-2d}{v}\right)_{4}\cdot 1=\left(\frac{x_\frak p+y_\frak p \sqrt{-1},\sqrt{-2d}}{v}\right)\\
& = \left(\frac{-1,2d}{2}\right)=-1
\endaligned$$
where the second equation holds by Lemma \ref{3-2d}.
\end{proof}

Using a similar argument as in the proof of Theorem \ref{0.1}, the following
result follows from Lemma \ref{3-2d} and \ref{2}.
\begin{prop}\label{-2p} Let $F=\Q(\sqrt{-2d})$ and $d\equiv 3 \mod 4$.
Then the diophantine equation $x^2+y^2=\alpha$ is solvable over
$\frak o_F$ if and only if $\alpha$ satisfies the Artin condition of
$\Theta$ and $H_L$, where $\Theta=E(\sqrt[4]{2})$ and $H_L$ is the
ring class field corresponding to the order $L=\frak o_{F}+ \frak
o_{F} \sqrt{-1}$.
\end{prop}

Now we use Proposition \ref{-2p}  to give an explicit example.
\begin{exa} Let
$F=\Q(\sqrt{-6})$. We write
$N_{F/\Q}(\alpha)=2^{s_1}3^{s_2}p_1^{e_1}\cdots p_g^{e_g}$ for any
$\alpha=a+b\sqrt{-6}\text{ and } a,b \in \Z$. Let $P(\alpha)=\{p_1,
\cdots, p_g \}$. Denote
$$\aligned& P_1=\{p\in P(\alpha): \left(\frac{-1}{p}\right)=\left(\frac{-6}{p}\right)=1 \text{ and }
\left(\frac{2}{p}\right)=-1 \}\cr & P_2= \{p\in P(\alpha):
\left(\frac{-1}{p}\right)=-\left(\frac{-6}{p}\right)=1 \text{ and } \left(\frac{2}{p}\right)=-1\}\\
& P_3= \{p\in P(\alpha):
\left(\frac{-1}{p}\right)=\left(\frac{-6}{p}\right)=1\text{ and
}\left(\frac{2}{p}\right)_4=-1\}.
\endaligned $$
It is easy to see that $e_i$ is even for $p_i\in P_2$.

Then $x^2+y^2=\alpha$ is solvable over $\frak
o_F$ if and only if
\begin{enumerate}[(1)]
\item The equation $x^2+y^2=\alpha$ has integral solutions at every
place of $F$.

\item $P_1\neq \emptyset$, or $2\mid a$, or $$\sum_{p_i\in
P_2}e_i/2+\sum_{p_i\in P_3}e_i\equiv \begin{cases} 0 \mod 2 \ \ \ &
\text{if }a \equiv 1, 3 \mod 8 \cr
 1\mod 2 \ \ \ & \text{if }a\equiv -1, -3 \mod 8 \end{cases}$$
for $P_1=\emptyset$ and $2\nmid a$.
\end{enumerate}
\end{exa}

\section{The sum of two squares in real quadratic fields}

Let $d> 1$ be a square-free odd number and $F=\Q(\sqrt{2d})$. Let
$\frak o_{F}$ be the ring of integers  of $F$, $\varepsilon_F$ the
fundamental unit of $\frak o_{F}$ and
$\varepsilon_F=a+b\sqrt{2d}\text{ with }a,b>0$. Let
$E=F(\sqrt{-1})$. One takes the order $L= \frak o_{F}+ \frak o_{F}
\sqrt{-1}$ inside $E$. In this section we always assume that one of
the equations $x^2-2dy^2=\pm 2$ is solvable over $\Z$ and  we fix
one solution $(x_0,y_0)$. Denote $\omega=x_0+y_0\sqrt{2d}$ and
$\eta=\omega^2/2$. Then $\eta\in \frak o_F^\times$ and
$\eta=\varepsilon_F^{i_0}$ for some $i_0\in \Z$. By the assumption,
we have $N_{F/\Q}(\varepsilon_F)=1$ (see \cite{Per}, pp. 106-109).

\begin{lem} \label{computation2} Let $p\equiv \pm 3 \mod 8 $ and $p|
d$. If one of the equations $x^2-2dy^2=\pm 2$ is solvable over $\Z$
with the notation as above, then $i_0$ is odd and  the $4$-th
Hilbert symbol
$$\left(\frac{x_\frak p+y_\frak p\sqrt{-1}, p}{v}\right)_{4}=-1$$
for $x_\frak p$ and $y_\frak p$ in $\frak o_{F_\frak p}$ satisfy
$x_\frak p^2+y_\frak p^2=\varepsilon_F$, where $v$ and $\frak p$ are
respectively  the unique prime in $E$ and $F$ above $p$.
\end{lem}
\begin{proof} The Hilbert symbol
$$\aligned\left(\frac{x_\frak p+y_\frak p\sqrt{-1}, p}{v}\right)_{4}
& = \left(\frac{x_\frak p+y_\frak p\sqrt{-1},2d}{v}\right)_{4}\cdot\left(\frac{x_\frak p+y_\frak p\sqrt{-1},2d/p}{v}\right)^{-1}_{4}\\
& = \left(\frac{x_\frak p+y_\frak p\sqrt{-1},2d}{v}\right)_{4}\cdot 1=\left(\frac{x_\frak p+y_\frak p \sqrt{-1},\sqrt{2d}}{v}\right)\\
& = \left(\frac{\varepsilon_F ,\sqrt{2d}}{\frak p}\right)
\endaligned$$
where the second equation holds since $E(\sqrt[4]{2d/p})/E$ is
unramified over $v$. However,$$\left(\frac{\eta ,\sqrt{2d}}{\frak
p}\right)=\left(\frac{\omega^2/2 ,\sqrt{2d}}{\frak
p}\right)=\left(\frac{2,p}{p}\right)=-1.$$ Then
$$-1=\left(\frac{\eta ,\sqrt{2d}}{\frak p}\right)=\left(\frac{\varepsilon_F,\sqrt{2d}}{\frak p}\right)^{i_0}.$$
Therefore $\left(\frac{\varepsilon_F,\sqrt{2d}}{\frak p}\right)=-1$
and $i_0$ is odd.
\end{proof}

\begin{lem} \label{3+d} Let $p|d$, $p\equiv \pm 3\mod
8$. Suppose one of the equations $x^2-2dy^2=\pm 2$ is solvable over $\Z$
with the notation as above. If $x_\frak p$ and $y_\frak p$ in $\frak
o_{F_\frak p}$ satisfy $x_\frak p^2+y_\frak p^2=\varepsilon_F$, then
$$\left(\frac{x_\frak p+y_\frak p\sqrt{-1}, p}{v}\right)_{4}=1$$
where $v$ and $\frak p$ are respectively the unique place of $E$ and
$F$ above $2$.
\end{lem}
\begin{proof}Since $E_v=F_\frak p(\sqrt{-1})$ is ramified over $F_\frak p$,
there is a uniformizer $\pi_F$ in $F_\frak p$ such that
$\left(\frac{\pi_F,-1}{\frak p}\right)=1$. We know $x^2+y^2=\pi_F^3$
is solvable over $\frak o_{F_\frak p}$ and $x^2+y^2=\pi_F$ is not
solvable over $\frak o_{F_\frak p}$ (by Theorem 1 in \cite{Ko}).
Choose $(a,b)$ be one solution of the equation $x^2+y^2=\pi_F^3$,
then we have $a\text{ and }b\in \frak o_{F_\frak p}^\times $,
otherwise we obtains $x^2+y^2=\pi_F$ is solvable over $\frak
o_{F_\frak p}$.

Recall $\omega=x_0+y_0\sqrt{2d}$ and $\eta=\omega^2/2$, here
$x_0,y_0\in \Z$ satisfy $x_0^2-2dy_0^2=\pm 2$. Denote
$$\epsilon=2^{-1}\omega
\cdot(1+\sqrt{-1})\cdot(a+b\sqrt{-1})/(a-b\sqrt{-1}).$$ We can see
$N_{E_v/F_\frak p}(\epsilon)=\eta$ and
$$\aligned\epsilon & =2^{-1}\pi_F^{-3}\omega\cdot(1+\sqrt{-1})\cdot(a^2-b^2+2ab\sqrt{-1})\\
& = 2^{-1}\pi_F^{-3}\omega\cdot[(a^2-b^2-2ab)+(a^2-b^2+2ab)\sqrt{-1}] \\
& =(\omega/2-
b\omega(a+b)/\pi_F^3)+(\omega/2-b\omega(b-a)/\pi_F^3)\sqrt{-1}.
\endaligned$$

In the following we will prove $a\not \equiv \pm b \mod 2$. If not
then $a=\pm b+2u$ for some $u\in \frak o_{F_\frak p}$. We have $(\pm
b+2u)^2+b^2=\pi_F^3.$ Note that $b\in \frak o_{F_\frak p}^\times$,
then
$$2=ord_{\frak p}(4u^2+2b^2\pm 4ub)=ord_{\frak p}(\pi_F^3)=3,$$
a contradiction is derived. So $a\not \equiv b \mod 2$. Since $a,b$
are units and the residue field of $\frak o_{F_\frak p}$ is $\Bbb
F_2$, we have
$$a+b=\pi_F u_1,a-b=\pi_F u_2$$ with $u_1,u_2\in \frak o_{F_\frak
p}^\times$. Then
$$\aligned\epsilon &=(\omega/2-bu_1\omega/\pi_F^2)+\sqrt{-1}(\omega/2+bu_2\omega/\pi_F^2)\\
&=2^{-1}\omega(1-bu_1\frac {2}{\pi_F^2})+2^{-1}\omega(1+bu_2\frac
{2}{\pi_F^2})\sqrt{-1}.\endaligned$$ We have $\epsilon \in \frak
o_{F_\frak p}+\frak o_{F_\frak p}\sqrt{-1}$ since $b,u_1,u_2\in
\frak o_{F_\frak p}^\times$ and the residue field of $\frak
o_{F_\frak p}$ is $\Bbb F_2$.

Let $\frak q$ be the unique prime of $\Q(\sqrt{-1})$ above $2$. Let
$s\in \Z_2^\times$ such that $s^2=p/3$ or $-p/3$. For any $\delta
\in E_v^*$ satisfies $N_{E_v/F_\frak p}(\delta)=\eta $, we have
$$\aligned\left(\frac{\delta, p}{v}\right)_{4}&=\left(\frac{\delta,3}{v}\right)_{4}\cdot \left(\frac{\delta,s}{v}\right)\cdot \left(\frac{\delta,\pm 1}{v}\right)_{4}
=\left(\frac{\delta,3}{v}\right)_{4}\cdot\left(\frac{\eta,s}{\frak
p}\right)\cdot1\\
&=\left(\frac{\delta,3}{v}\right)_{4}\cdot\left(\frac{1,s}{
p}\right)=\left(\frac{\delta,3}{v}\right)_{4}.
\endaligned$$
Then
$$\aligned\left(\frac{\epsilon, p}{v}\right)_{4}& = \left(\frac{2^{-1}\pi_F^{-3}\omega\cdot(1+\sqrt{-1})\cdot(a+b\sqrt{-1})^2,3}{v}\right)_{4}\\
&=\left(\frac{2^{-1},3}{v}\right)_{4}\cdot\left(\frac{\pi_F^3,3}{v}\right)^{-1}_{4}\cdot \left(\frac{\omega,3}{v}\right)_{4}\cdot\left(\frac{1+\sqrt{-1},3}{v}\right)_{4}\cdot\left(\frac{a+b\sqrt{-1},3}{v}\right)\\
&=\left(\frac{2^{-1},3}{\frak q}\right)\cdot
\left(\frac{\pi_F^3,3}{v}\right)^{-1}_{4}\cdot \left(\frac{\pm
2,3}{\frak q}\right)_4\cdot \left(\frac{1+\sqrt{-1},3}{\frak
q}\right)\cdot
\left(\frac{\pi_F^{3},3}{\frak p}\right)\\
&=1\cdot \left(\frac{\pi_F^3,3}{v}\right)^{-1}_{4}\cdot 1\cdot(-1) \cdot\left(\frac{\pi_F,3}{\frak p}\right)\\
\endaligned$$
Since $\left(\frac{\pi_F,-1}{\frak p}\right)=1$ by the choice of
$\pi_F$ and $F_\frak p(\sqrt{-3})/F_\frak p$ is unramifed of degree
$2$, then one obtains $$\left(\frac{\pi_F,3}{\frak
p}\right)=\left(\frac{\pi_F,-3}{\frak p}\right)=-1.$$ Therefore
$$\aligned\left(\frac{\epsilon, p}{v}\right)_{4}&=\left(\frac{\pi_F^3,3}{v}\right)^{-1}_{4}
= \left(\frac{N_{E_v/\Q_2}(a+b\sqrt{-1}),3}{\frak
q}\right)^{-1}_{4}.\endaligned$$ We can write
$N_{E_v/\Q_2}(a+b\sqrt{-1})=2^3m$ with $m\equiv 1 \text{ or } -3\mod
8$. Therefore
$$\aligned\left(\frac{\epsilon, p}{v}\right)_{4}&=\left(\frac{2^3m,3}{\frak
q}\right)_{4}^{-1}=\left(\frac{m,3}{\frak q}\right)_{4}^{-1}\\
=&\begin{cases}1 &\text{ if } m\equiv 1\mod
8\\
\left(\frac{-3,3}{\frak q}\right)_4^{-1}=1 &\text{ if } m\equiv
-3\mod 8.
\end{cases}\endaligned$$
Let $x_\frak p$ and $y_\frak p$ in $\frak o_{F_\frak p}$ satisfy
$x_\frak p^2+y_\frak p^2=\varepsilon_F$. Since
$\eta=\varepsilon_F^{i_0}$, we have
$$\left(\frac{\epsilon,p}{v}\right)_{4}=\left(\frac{x_\frak p+y_\frak p\sqrt{-1},
p}{v}\right)_{4}^{i_0}$$ by Lemma \ref{3-2d}. Since $i_0$ is odd by
Lemma \ref{computation2}, one obtains $$\left(\frac{x_\frak
p+y_\frak p\sqrt{-1}, p}{v}\right)_{4}=1.$$
\end{proof}

\begin{cor} \label{er} Let $F=\Q(\sqrt{2d})$ and $p|d$, $p\equiv \pm 3\mod 8$.
If one of the equations $x^2-2dy^2=\pm 2$ is solvable over $\Z$,
then $x^2+y^2=\varepsilon_F$ is not solvable over $\frak o_F$.
\end{cor}
\begin{proof} Recall $\bold X_{\varepsilon_F}$ is the affine scheme over $\frak o_F$ defined by
$x^2+y^2=\varepsilon_F$. Let $\Theta=E(\sqrt[4]{p})$. One obtains $$
\psi_{\Theta/E}(f_E[\prod_{\frak p\leq \infty}(x_\frak p,y_\frak
p)])=\left(\frac{x_{\frak p'}+y_{\frak p'}\sqrt{-1},
p}{v}\right)_{4}=-1
$$ for any $\prod_{\frak p\leq \infty}(x_\frak p,y_\frak p)\in \prod_{\frak p\leq \infty}
\bold X_{\varepsilon_F}(\frak o_{F_\frak p})$ by Lemma
\ref{computation2} and \ref{3+d}, where $\psi_{\Theta/E}$ is the
Artin map, $v$ and $\frak p'$ are respectively the unique prime of
$E$ and $F$ above $p$ . The result follows from the class field
theory.
\end{proof}
\begin{rem*} In fact $x^2+y^2=\varepsilon_F$ has local integral solutions at every
place of $F$. The solvability is obvious if the place is not above
$2$. Let $\frak p$ be the unique place of $F$ above 2. Then we only
need to show that it is solvable over $\frak o_{F_\frak p}$. Since
$\eta=\omega^2/2=\varepsilon_F^{i_0}$ with $i_0$ odd, we have
$$x^2+y^2=\varepsilon_F \text{ is solvable over } \frak o_{F_\frak p} \Leftrightarrow x^2+y^2=\eta \text{ is solvable over } \frak o_{F_\frak p}.$$
Note that $\omega=x_0+y_0\sqrt{2d}$ with $2\mid x_0$ and $y_0$ odd.
It is easy to verify that $x^2+y^2=\eta$ is solvable over $\frak
o_{F_\frak p}$ by Theorem 1 in \cite{Ko}.
\end{rem*}

\begin{thm}\label{thm-2d} Let $F=\Q(\sqrt{2d})$. If $p|d$, $p\equiv \pm 3\mod
8$ and one of the equations $x^2-2dy^2=\pm 2$ is solvable over $\Z$,
then the diophantine equation $x^2+y^2=\alpha$ is solvable over
$\frak o_F$ if and only if $\alpha$ satisfies the Artin condition of
$\Theta$ and $H_L$, where $\Theta=E(\sqrt[4]{p})$ and $H_L$ is the
ring class field corresponding to the order $L=\frak o_{F}+ \frak
o_{F} \sqrt{-1}$.
\end{thm}
\begin{proof}
Let $\frak p$ be a place of $F$ and $L_\frak p$ the $\frak p$-adic
completion of $L$ inside $E_\frak p = E\otimes_F F_\frak p$.

Let $v_1$  and $\frak p_1$  be respectively the unique prime of $E$
and $F$ above $2$. By Lemma \ref{3-2d}, one has $\left(\frac{\xi,
p}{v}\right)_{4}=1$ for any $\xi\in L_{\frak p_1}^\times $ with
$N_{E_{v_1}/F_{\frak p_1}}(\xi)=1$. This implies that
$$\lambda_E(\bold T(\frak o_{F_{\frak p_1}}))\subseteq N_{\Theta_{\frak V_1}/E_{v_1}}(\Theta_{\frak
V_1}^*) $$ where $\frak V_1$ is the unique place of $\Theta$ above
$v_1$.

Let $v_2$  and $\frak p_2$  be respectively the unique prime of $E$
and $F$ above $p$. By the similar computation in Lemma
\ref{computation2}, one has $\left(\frac{\xi, p}{v}\right)_{4}=1$
for any $\xi\in L_{\frak p_2}^\times $ with $N_{E_{v_2}/F_{\frak
p_2}}(\xi)=1$. This implies that
$$\lambda_E(\bold T(\frak o_{F_{\frak p_2}}))\subseteq N_{\Theta_{v_2}/E_{v_2}}(\Theta_{v_2}^*) $$ where $\Theta_{v_2}=\Theta\otimes_E E_{v_2}$.

Since $\Theta/E$ is unramified over all primes except $v_1$ and
$v_2$,  therefore the natural group homomorphism
$$\widetilde{\lambda}_E: \nicefrac{T(\Bbb A_F)}{T(F)\prod_{\frak p\leq \infty}\bold T(\frak o_{F_\frak p})}
\longrightarrow [\nicefrac{\Bbb I_E}{E^* N_{\Theta/E}(\Bbb
I_{\Theta})}] \times [\nicefrac{\Bbb I_E}{E^* \prod_{\frak p\leq
\infty} L_\frak p^\times}]$$ is well-defined. By Proposition
\ref{multiple}, we only need to show $\widetilde{\lambda}_E$ is
injective.

Let $u\in ker \widetilde{\lambda}_E$. Then there are $\beta\in E^*$
and $i\in \prod_{\frak p\leq \infty} L_\frak p^\times$ with
$\widetilde{\lambda}_E(u)=\beta i$. Therefore
$$N_{E/F}(\beta)=N_{E/F}(i)^{-1} \in F^*\cap
(\prod_{\frak p\leq \infty}\frak o_{F_\frak p}^\times )=\{\pm
\varepsilon_F^n|n\in \Z\}.$$ Since $N_{E/F}(\beta)$ is totally
positive, we have $N_{E/F}(\beta)=\varepsilon_F^n$ and
$N_{E/F}(i)=\varepsilon_F^{-n}$.

Assume that $n$ is odd. Write $i=(i_v)_{v}\in \Bbb I_E$. Since
$\Theta/E$ is unramified over all primes of $E$ except $v_1$ and
$v_2$, one has $\psi_{\Theta/E}(i_v)$ is trivial for all primes
$v\neq v_1,v_2$, where $i_v$ is regarded as an idele whose
$v$-component is $i_v$ and 1 otherwise. Since $N_{E_{v}/F_\frak
p}(i_{v})=\varepsilon_F^{-n},$ one gets
$$\psi_{\Theta/E}(\beta i)=\psi_{\Theta/E}(i)=\psi_{\Theta/E}(i_{v_1})\psi_{\Theta/E}(i_{v_2})=1\cdot(-1)^{-n}=-1$$
by Lemma \ref{computation2} and \ref{3+d}, where $\psi_{\Theta/E}$
is the Artin map. This contradicts to $u \in ker
\widetilde{\lambda}_E$.

Therefore $n$ is even. Let $$\gamma=\beta
\varepsilon_F^{n/2},j=i\varepsilon_F^{-n/2}.$$ Then
$$ N_{E/F}(\gamma) =N_{E/F}(j) =1 \ \ \ \Rightarrow \ \ \ \gamma\in
T(F) \ \ \ \text{and} \ \ \ j\in \prod_{\frak p\leq \infty}\bold
T(\frak o_{F_\frak p}).$$ So $\beta i=\gamma j\in T(F)\prod_{\frak
p\leq \infty}\bold T(\frak o_{F_\frak p})$. Then
$\widetilde{\lambda}_E$ is injective.
\end{proof}

In the following we consider a special case. Let $d=p$ be a prime
with $p\equiv 3 \mod 8$. Then the equation $x^2-2py^2=-2$ is
solvable over $\Z$ (Corollary 2 in \cite{Yo}).
\begin{lem} \label{2-3} Let $F=\Q(\sqrt{2p})$ and $p\equiv 3 \mod 8$.
If $x_\frak p$ and $y_\frak p$ in $\frak o_{F_\frak p}$ satisfy
$x_\frak p^2+y_\frak p^2=\varepsilon_F$, then the $4$-th Hilbert
symbol
$$\left(\frac{x_\frak p+y_\frak p\sqrt{-1}, 2}{v}\right)_{4}=-1$$
where $v$ and $\frak p$ are respectively  the unique prime in $E$
and $F$ above $2$.
\end{lem}
\begin{proof} The Hilbert symbol
$$\aligned\left(\frac{x_\frak p+y_\frak p\sqrt{-1}, 2}{v}\right)_{4}
& = \left(\frac{x_\frak p+y_\frak p\sqrt{-1},2p}{v}\right)_{4}\cdot\left(\frac{x_\frak p+y_\frak p\sqrt{-1},p}{v}\right)^{-1}_{4}\\
& = \left(\frac{x_\frak p+y_\frak
p\sqrt{-1},\sqrt{2p}}{v}\right)\cdot
1=\left(\frac{\varepsilon_F,\sqrt{2p}}{\frak p}\right)\endaligned$$
where the second equation holds by Lemma \ref{3+d}. Recall
$\omega=x_0+y_0\sqrt{2d}$ and $\eta=\omega^2/2$, here $x_0,y_0\in
\Z$ satisfy $x_0^2-2py_0^2=-2$. Since $\eta=\varepsilon_F^{i_0}$ for
some $i_0\in \Z$, we have $i_0$ is odd by Lemma \ref{computation2}.
So we have
$$\aligned\left(\frac{x_\frak p+y_\frak p\sqrt{-1}, 2}{v}\right)_{4}&=
\left(\frac{\varepsilon_F,\sqrt{2p}}{\frak p}\right)=\left(\frac{\omega^2/2,\sqrt{2p}}{\frak p}\right)\\
&=\left(\frac{2,\sqrt{2p}}{\frak
p}\right)=\left(\frac{2,-2p}{2}\right)=-1.
\endaligned$$
\end{proof}

Using a similar argument as in the proof of Theorem \ref{thm-2d},
the result follows from Lemma \ref{3+d} and \ref{2-3}.
\begin{prop}\label{d} Let $p$ be a prime and $F=\Q(\sqrt{2p})$. If $p\equiv 3 \mod 8$,
then the diophantine equation $x^2+y^2=\alpha$ is solvable over
$\frak o_F$ if and only if $\alpha$ satisfies the Artin condition of
$\Theta$ and $H_L$, where $\Theta=E(\sqrt[4]{2})$ and $H_L$ is the
ring class field corresponding to the order $L=\frak o_{F}+ \frak
o_{F} \sqrt{-1}$.
\end{prop}

Now we use Proposition  \ref{d} to give an explicit example.

\begin{exa} Let
$F=\Q(\sqrt{6})$. We write
$N_{F/\Q}(\alpha)=2^{s_1}3^{s_2}p_1^{e_1}\cdots p_g^{e_g}$ for any
$\alpha=a+b\sqrt{6}$, here $a,b \in \Z$. Let $D(\alpha)=\{p_1, \cdots,
p_g \}$. Denote
$$\aligned& D_1=\{p\in D(\alpha): \left(\frac{-1}{p}\right)=\left(\frac{6}{p}\right)=1 \text{ and }
\left(\frac{2}{p}\right)=-1 \}\cr & D_2= \{p\in D(\alpha):
\left(\frac{-1}{p}\right)=-\left(\frac{6}{p}\right)=1 \text{ and } \left(\frac{2}{p}\right)=-1\}\\
& D_3= \{p\in D(\alpha):
\left(\frac{-1}{p}\right)=\left(\frac{6}{p}\right)=1\text{ and
}\left(\frac{2}{p}\right)_4=-1\}.
\endaligned $$
It is easy to see that $e_i$ is even for $p_i\in D_2$.

Then $x^2+y^2=\alpha$ is solvable over $\frak
o_F$ if and only if
\begin{enumerate}[(1)]
\item The equation $x^2+y^2=\alpha$ has integral solutions at every
place of $F$.

\item $D_1\neq \emptyset$, or $2\mid a$, or
$$\sum_{p_i\in D_2}e_i/2+\sum_{p_i\in
D_3}e_i\equiv \begin{cases} 0 \mod 2 \ \ \ & \text{if }a\equiv \pm
1 \mod 8 \cr
 1\mod 2 \ \ \ & \text{if }a\equiv \pm 3 \mod 8 \end{cases}$$
for $D_1=\emptyset$ and $2\nmid a$.
\end{enumerate}
\end{exa}

\bf{Acknowledgment} \it{The author thanks the referee for many useful suggestions. The work is supported by the Morningside
Center of Mathematics and grant DE 1646/2-1 of the Deutsche Forschungsgemeinschaft. The author is supported by  NSFC, grant \# 10901150 and 973 Program 2013CB834202.}

%-----------------------------------------------------------------------%

\end{document}